\documentclass[12pt]{article}
\usepackage{amsmath}
\usepackage{amssymb}
\usepackage{amsthm}

\newtheorem{thm}{Theorem}
\newtheorem{prop}[thm]{Proposition}

\def\R{\mathbb{R}}

\def\epar{\partial_{\varepsilon}}
\def\pa{\partial }
\def\dom{\mathrm{dom}\,}

\def\e{\varepsilon}

\begin{document}

\title{Maximal monotonicity of the subdifferential\\of a
convex function: a direct proof\thanks{Dedicated to Prof. R. T. Rockafellar on the occasion of his 80th anniversary.}}

\author{Milen Ivanov\thanks{Partially supported by Bulgarian National Scientific Fund under Grant DFNI-I02/10 and by
Grant 239/2016 of the  Science Fund of Sofia University.}       \and
        Nadia Zlateva\thanks{Partially supported by Bulgarian National Scientific Fund under Grant DFNI-I02/10.}}

\date{May, 2015}

\maketitle

\begin{abstract}
We provide a new proof for maximal monotonicity of the
subdifferential of a convex function.
\end{abstract}

\textbf{Keywords:}  Subdifferential, convex function, maximal monotonicity.

\textbf{Subject Classification} 49J52,   47H05,  49J45.

\section{Introduction}
\label{intro}
Let $X$ be a Banach space with dual $X^*$. Let $f:X\to\R\cup\left\{ +\infty\right\} $
be a proper, convex and lower semicontinuous function.

Recall that for $x_{0}\in\dom f:=\left\{ x:\, f(x)\neq +\infty\right\} $ the subdifferential of $f$ at $x_0$ in the sense of convex analysis is the set
\begin{equation}
\partial f(x_{0}):=\left\{ p\in X^{*}:\, f(x)\ge f(x_{0})+p(x-x_{0}),\,\forall x\in X\right\} ,\label{eq:def_subdif}
\end{equation}
while the $\e$-subdifferential of $f$ at $x_0$ is the set
\[
\partial_{\varepsilon}f(x_0):=\left\{ p\in X^{*}:\, f(x)\ge f(x_{0})+p(x-x_{0})-\varepsilon,\,\forall x\in X\right\} .
\]
If $x_{0}\not\in\dom f$, then $\partial f(x_{0})=\epar f(x_0)=\emptyset.$

For $x_0\in \dom f$ the set $\partial f(x_0)$ might be empty but the set
$\epar f(x_0)$ is always non-empty  for any $\varepsilon >0$.

It is clear that $\partial f$ is a \emph{monotone multivalued} mapping from $X$ to $X^*$, in the sense that
for all $x,\ y\in X$ and all $p\in \pa f(x)$, $q\in \pa f(y)$ it holds
$$
  \langle q-p, y-x \rangle \ge 0,
$$
which can be written also as
$$
  \langle \pa f(y)-\pa f(x), y-x \rangle \ge 0,\quad \forall x,y\in X.
$$

It is a well-known classical result due to Rockafellar that $\partial f$ is in fact \emph{maximal} monotone (see Theorem~\ref{mainthm}).
That is, $\partial f$ can not be extended to strictly larger monotone mapping from $X$ to $X^*$.

This statement goes back at least as far as \cite{minty}, where Minty proves it for a continuous convex function on Hilbert space
(see the discussion in \cite{rock66}).

Moreau \cite{moreau} gave a proof in Hilbert space using duality and Moreau-Yosida approximation.

A stumbling block to generalizing Minty's method is that for a lower semicontinuous convex
function which is not continuous, $\partial f$ might be empty at some points.
Also controlling the norms of the subgradients appearing in the proof is not trivial (see the discussion in \cite{rock70}).
On the other hand, the method of Moreau relies heavily on the fact that Hilbert space is canonicaly isometric to its dual.

The method of reducing the considerations to a line, used by Minty in \cite{minty} and generalised by Rockafellar in \cite{rock66}, is prevalent in the
subsequent proofs, like those of Taylor \cite{taylor}, Borwein \cite{borwein}, Thibault \cite{thibault}, the unpublished proof of
Zagrodny, using \cite{zagrodni}, and the recent one of Jules and Lasonde \cite{ju_lasond}. A textbook following this line of proof
is \cite{phelps1}.

The first complete proof in Banach space: that of Rockafellar \cite{rock70}, is a methodological break through showing that Fenchel conjugate and
duality can be used in non-reflexive case as well. The  recent proofs of Marques Alves and Svaiter \cite{alves_svatier} and
Simons \cite{simons_new} also use duality. It is mentioned in \cite{zalinescu} that in many textbooks authors prefer proving only the reflexive
case (where duality techniques are easier due to symmetry), e.g. \cite[p. 278]{zalinescu}.

The famous proof of Simons \cite{simons_lemma} (see also \cite{phelps2}) shows how one can pick a subgradient with controlled norm.

Like some others, our proof starts with
\[
\left\langle \partial f(x),x\right\rangle \ge 0,\quad\forall x\in \dom \pa f.
\]
 which is a sufficient condition of minimality of Minty type (see \cite{bocko}).

We prove it by adding a slack function to ensure that the sum is bounded below.

Finally, all tools we use had been well-known by 1970.

\section{Proof of the main result}
The following is proved in \cite{ju_lasond} for general lower semicontinuous function
through mean value inequality. We give a simple proof for the convex case.

\begin{prop}\label{prop:1}
Let $X$ be a Banach space and let $f:X\to\R\cup\left\{ +\infty\right\} $
be a proper, convex and lower semicontinuous function.

If $f:X\to\R\cup\left\{ +\infty\right\} $ satisfies
\begin{equation}\label{eq:0_mon}
\left\langle \partial f(x),x\right\rangle \ge 0,\quad\forall x\in \dom \pa f
\end{equation}
then $0\in\partial f(0)$.
\end{prop}

\begin{proof}
  For $a>0$ let $g_a(x) := a\|x\|^2$.

  If $p\in\partial g_a(x)$ then by definition $p(0-x)\le g_a(0) - g_a(x)$, that is,
  \begin{equation}\label{new:b}
    p(x) \ge a\|x\|^2.
  \end{equation}
  Let
  \begin{equation}\label{new:c}
    f_a(x) := f(x) + g_a(x).
  \end{equation}
  Since $g_a$ is continuous, the Sum Theorem, see for example \cite{zalinescu}, implies that $\dom \partial f_a = \dom \partial f$ and
  \begin{equation}\label{new:d}
    \partial f_a(x) = \partial f(x) + \partial g_a(x),\quad \forall x\in\dom\partial f.
  \end{equation}
  From \eqref{eq:0_mon}, \eqref{new:b} and \eqref{new:d} it follows that
  $$
    \langle\partial f_a (x),x\rangle \ge a\|x\|^2,\quad\forall x\in\dom\partial f.
  $$
  Consequently,
  \begin{equation}\label{new:e}
    \forall p\in\partial f_a(x) \Rightarrow \|p\|\ge a\|x\|.
  \end{equation}
  On the other hand, $f_a$ is bounded below for each $a > 0$. Indeed, take $x_0\in\dom f$. Since $f$ is
  lower semicontinuous there is $\delta > 0$ such that $\inf f(x_0+\delta B_X) > -\infty$, where $B_X$ is the closed unit ball.
  Take $c\in\mathbb{R}$ such that $c<\inf f(x_0+\delta B_X)$.
  By Hahn-Banach Theorem, see for example \cite{holmes}, we can separate the epigraph of $f$, that is the set $\{(x,t):\ f(x)\le t\}$, from the set $(x_0+\delta B_X)\times(-\infty,c]$. So, there is  $(p,r)\in X^*\times\mathbb{R}\setminus (0,0)$
  such that
  $$
    p(x) + rt \ge p(y) + rs,\quad \forall x\in\dom f,\ \forall t\ge f(x),\ \forall y\in x_0+\delta B_X,\ \forall s \le c.
  $$
  Setting $x=x_0$ we see that $\sup\{rs:\ s\le c\}<\infty$ which is only possible if $r\ge 0$. But if $r=0$ then setting $x=x_0$ and $y=x_0+h$ we see that $p(x_0)\ge p(x_0)+p(h)$ for all $h\in\delta B_X$ which implies $p=0$, contradiction. Therefore, $r>0$ and we can divide the above inequality by $r$ as well use $t=f(x)$ and $s=c$ to obtain for $q=p/r$ that
  $$
    q (x) + f(x) \ge q(y) + c,\quad \forall x\in\dom f,\ \forall y\in x_0+\delta B_X,
  $$
  Set $y=x_0$ and rearrange to get for $r= q(x_0)+c$
  $$
    f(x) \ge - q(x) +r \ge -\|q\|\|x\| +r,\quad \forall x\in X.
  $$ 
  Therefore, $f_a(x) \ge r + \|x\|(a\|x\|-\|q\|)$ and $f_a$ is bounded below.

  Let $x_n$ be a minimising sequence for $f_a$. So, $f_a(x_n) < f_a(x_n)+\varepsilon_n$ for some
  $\varepsilon_n\to 0$. Equivalently, $0\in\partial_{\varepsilon_n}f_a(x_n)$.

  From Br\o ndsted-Rockafellar Theorem, see \cite{BRO_ROCK}, it follows that there are
  $y_n\in x_n+\sqrt{\varepsilon_n} B_X$ and $p_n\in\partial f_a (y_n)$ such that
  $\|p_n\| \le \sqrt{\varepsilon_n}$. From this and \eqref{new:e} it follows that
  $$
    \|y_n\|\le\frac{\sqrt{\varepsilon_n}}{a}.
  $$
  Therefore, $x_n\to 0$. Since $f_a$ is lower semicontinuous, $0$ is the global
  minimum of $f_a$.

  In other words,
  $$
    f_a(x) \ge f_a(0) \iff f(x)\ge f(0) - a\|x\|^2,\quad\forall x\in X.
  $$
  Since $a>0$ was arbitrary, $0$ is a global minimum of $f$, or, equivalently,
  $0\in\partial f(0)$.
\end{proof}

The Rockafellar's Theorem follows by an easy and well known argument (see for example \cite{phelps1},  p. 59):

\begin{thm}\label{mainthm}
(Rockafellar~\cite{rock70}) Let X be a Banach space and let $f:X\to\R\cup\left\{ +\infty\right\} $ be a proper, convex
and lower semicontinuous function. Then $\partial f$ is a maximal monotone mapping
from $X$ to $X^{*}$.\end{thm}

\begin{proof}
Let $(y,q)\in X\times X^{*}$ be in monotone relation  to the
graph of $\partial f$, that is
\begin{equation}\label{eq:mon_rel}
\left\langle \partial f(x)-q,x-y\right\rangle \ge0,\quad\forall x\in \dom \pa f.
\end{equation}
Consider the function
\[
\bar{f}(x):=f(x+y)-q(x).
\]
It is immediate to check
that (\ref{eq:mon_rel}) implies (\ref{eq:0_mon}) for $\bar{f}$.
By Proposition~\ref{prop:1} we get $0\in\partial\bar{f}(0)$
which easily translates to $q\in\partial f(y)$. Therefore, $\partial f$
cannot be properly  extended in a monotone way.
\end{proof}

\end{document}